\documentclass[12pt,reqno]{amsart}

\bibliography{references}

\usepackage{times}
\usepackage[T1]{fontenc}
\usepackage{mathrsfs}
\usepackage{latexsym}
\usepackage[dvips]{graphics}
\usepackage[dvips]{graphicx}
\usepackage{epsfig}
\usepackage{amsmath,amsfonts,amsthm,amssymb,amscd,cleveref}
\usepackage{color}
\usepackage{cite}

\numberwithin{equation}{section}

\theoremstyle{plain}
\newtheorem{theorem}{Theorem}[section]
\newtheorem{thm}[theorem]{Theorem}
\newtheorem{prop}[theorem]{Proposition}

\newtheorem{lem}[theorem]{Lemma}

\theoremstyle{definition}
\newtheorem{defn}[theorem]{Definition}

\newtheorem{eg}[theorem]{Example}

\theoremstyle{remark}

\newcommand{\F}{\mathbb{F}}

\numberwithin{equation}{section}

\newcommand{\nin}{\notin}

\newcommand{\isom}{\cong}

\newcommand{\Z}{\mathbb{Z}}
\newcommand{\Q}{\mathbb{Q}}

\newcommand{\N}{\mathbb{N}}

\newcommand{\foh}{\frac{1}{2}}

\newcommand{\comment}[1]{}

\newcommand{\Mod}[1]{\,\,\left(\operatorname{mod}\, #1\right)}

\newcommand{\abs}[1]{\left| #1 \right|}

\DeclareFontFamily{U}{wncy}{}
\DeclareFontShape{U}{wncy}{m}{n}{<->wncyr10}{}
\DeclareSymbolFont{mcy}{U}{wncy}{m}{n}
\DeclareMathSymbol{\Sha}{\mathord}{mcy}{"58}

\begin{document}

\title{On $2$-superirreducible polynomials over finite fields}

\author[Bober]{J. W. Bober}
\address{School of Mathematics, University of Bristol, 
Fry Building, Woodland Road, Bristol BS8 1UG, UK and the Heilbronn Institute for Mathematical Research, Bristol, UK}
\email{j.bober@bristol.ac.uk}

\author[Du]{L. Du}
\address{Department of Mathematics, Draper Building, Berea College, 101 Chestnut St., Berea, KY 40404, USA}
\email{dul@berea.edu}

\author[Fretwell]{D. Fretwell}
\address{Department of Mathematics and Statistics, Fylde College, Lancaster University, Lancaster LA1 4YF, UK}
\email{d.fretwell@lancaster.ac.uk}

\author[Kopp]{G. S. Kopp}
\address{Department of Mathematics, Louisiana State University, Lockett Hall, Baton Rouge, LA 70803, USA}
\email{kopp@math.lsu.edu}

\author[Wooley]{T. D. Wooley}
\address{Department of Mathematics, Purdue University, 150 N. University Street, West 
Lafayette, IN 47907-2067, USA}
\email{twooley@purdue.edu}

\subjclass[2020]
{11T06, 
12E05, 
11S05
}

\keywords{Irreducibility, finite fields, polynomial compositions}

\begin{abstract} 
We investigate $k$-superirreducible polynomials, by which we mean irreducible 
polynomials that remain irreducible under any polynomial substitution of positive 
degree at most $k$. Let $\mathbb F$ be a finite field of characteristic $p$. We show 
that no $2$-superirreducible polynomials exist in $\mathbb F[t]$ when $p=2$ and that 
no such polynomials of odd degree exist when $p$ is odd. We address the remaining case 
in which $p$ is odd and the polynomials have even degree by giving an explicit formula 
for the number of monic 2-superirreducible polynomials having even degree $d$. This 
formula is analogous to that given by Gauss for the number of monic irreducible 
polynomials of given degree over a finite field. We discuss the associated asymptotic 
behaviour when either the degree of the polynomial or the size of the finite field 
tends to infinity.
\end{abstract}

\maketitle

\section{Introduction}
Superirreducible polynomials are polynomials that resist factorization under polynomial 
substitutions. Let $R$ be a commutative domain with unity having field of fractions $F$, 
and consider a polynomial $f\in R[t]$. For each natural number $k$, we say that $f$ is 
{\it weakly $k$-superirreducible} over $R$ if $f(g(t))$ is irreducible over $F[t]$ for 
all polynomials $g\in R[t]$ having degree $k$. If the polynomial $f$ is weakly 
$k$-superirreducible over $R$ for $1\leq k\leq K$, then we say that $f$ is 
{\it $K$-superirreducible}. In this hierarchy, a $1$-superirreducible polynomial is simply 
an irreducible polynomial. It transpires that a polynomial may be weakly 
$k$-superirreducible and yet not weakly $(k-1)$-superirreducible (and consequently, not 
$k$-superirreducible). For example, one may check that $x^6+x^5+x^3+x^2+1$ is weakly 
$3$-superirreducible over $\F_2$ yet not weakly $2$-superirreducible. When $d\geq 2$ and 
$f\in R[t]$ has degree $d$, a consideration of the polynomial $f(t+f(t))$ reveals that $f$ 
cannot be $k$-superirreducible whenever $k\geq d$. The situation when $2\leq k<d$ is more 
subtle, however, and our focus in this paper lies on the simplest situation here in which 
$k=2$ and $R$ is a finite field. Let $\mathbb F_q$ denote the finite field of $q$ elements, 
and when $1\leq k<d$, denote by $s_k(q,d)$ the number of monic weakly $k$-superirreducible 
polynomials lying in $\mathbb F_q[t]$ having degree $d$. In \Cref{prop:exactcount}, we 
provide an explicit formula for $s_2(q,d)$ analogous to the formula given by Gauss for the 
number of monic irreducible polynomials of given degree over $\mathbb F_q$. A consequence of 
this formula delivers the asymptotic formula recorded in our first theorem.

\begin{thm}\label{theorem1.1} For $q$ a prime power and $d$ a positive integer, the number of $2$-superirreducible polynomials of degree $d$ satisfies: 
\begin{enumerate}
    \item[(a)] $s_2(q,d)=0$ whenever either $q$ is a power of $2$ or $d$ is odd;
    \item[(b)] $s_2(q,d)=0$ whenever $q>(d-1)^2$;
    \item[(c)] when $q$ is odd and $d\rightarrow \infty$ through the even integers,
    \begin{equation}\label{1.1}
    s_2(q,d)=\frac{q^d}{d2^q}+O\!\left(\frac{1}{d}q^{d/2}\right). 
    \end{equation}
\end{enumerate}
\end{thm}

Superirreducibility has in fact been studied in the past, although not by name. Strengthening the 
above pedestrian observation concerning $f(t+f(t))$, it follows from work of Schinzel 
\cite[Lemma 10]{schinzel} that a polynomial of degree $d\geq 3$ lying in $\mathbb Z[t]$ cannot be 
$(d-1)$-superirreducible. More recently, Bober et al.~\cite {BFMW} have considered 
superirreducibility as a potential limitation to the understanding of smooth integral values of 
polynomials. More precisely, these authors show in \cite[Theorem 1.1]{BFMW} that quadratic 
polynomials $f\in \Z[t]$ admit polynomial substitutions $g\in \Z[t]$ of arbitrarily high degree 
$k$ having the property that all of the factors of $f(g(t))$ have degree 
$O(k/\sqrt{\log \log k})$. Consider a positive number $\varepsilon$ and an integer $k$ 
sufficiently large in terms of $\varepsilon$. Then on taking $m$ to be an integer large enough in 
terms of both $\varepsilon$ and $k$, it follows that with $n=g(m)$, the polynomial value $f(n)$ 
has all of its prime factors smaller than $n^\varepsilon$ \cite[Corollary 1.2]{BFMW}. This 
provides strong information about smooth values $f(n)$ for $n\in \Z$. It is thus important to 
understand polynomials $f$ (in degree higher than $2$) that resist such factorizations of 
compositions $f(g(t))$, since the existence of smooth values of such polynomials will necessarily 
be particularly challenging to establish. In \cite[Section 6]{BFMW}, it is shown that 
$2$-superirreducible polynomials exist in $\mathbb Q[t]$ having degree $6$. Moreover, in work 
contemporaneous with that reported on herein, the thesis of Du \cite[Theorem 1.9.1]{thesis} has 
exhibited some infinite families of $2$-superirreducible polynomials in $\mathbb Q[t]$ of degree 
$4$, including the simple examples $f(t)=t^4+1$ and $f(t)=t^4+2$.\par

With a potential local-global principle in mind, it might be expected that insights 
into the superirreducibility of polynomials over $\Z$ and over $\Q$ might be gained by 
examining corresponding superirreducibility properties over the $p$-adic integers 
$\Z_p$ and $p$-adic numbers $\Q_p$. Such considerations lead in turn to an 
investigation of the superirreducibility of polynomials over finite fields. We finish 
our paper by disappointing the reader in \Cref{sec:4} with the news that if $k\ge 2$ and $p$ 
is any prime number, then $k$-superirreducible polynomials exist over neither $\Z_p$ 
nor $\Q_p$.

Other authors have considered arithmetic properties of compositions, especially compositional iterates, of polynomials within the context of arithmetic dynamics. See the papers \cite{jonesandlevy,odoni} and the survey \cite[Section 19]{benedetto} for a wide variety of results and questions concerning irreducibility of polynomial iterates and composites.

\section{Basic lemmas}
In this section we prove the basic lemmas that provide the infrastructure for our 
subsequent discussions concerning superirreducibility. Recall the definition of 
$k$-superirreducibility provided in our opening paragraph. We begin by expanding on 
the observation that there are no weakly $k$-superirreducible polynomials of degree $k$ or larger.

\begin{lem}\label{highk}
Let $R$ be a commutative domain with unity, and let $f\in R[t]$ be a polynomial of 
degree $d\ge 2$. Then $f(t)$ is not weakly $k$-superirreducible for any $k\geq d$.
\end{lem}

\begin{proof}
For each non-negative integer $r$, consider the degree $d+r$ substitution 
$g(t)=t+t^rf(t)$. We have
\begin{equation*}
    f(g(t))=f(t+t^rf(t))\equiv f(t)\equiv 0 \Mod{f(t)}.
\end{equation*}
Thus, we see that $f(g(t))$ is divisible by $f(t)$, and it is hence reducible. It follows 
that $f$ is not weakly $k$-superirreducible for $k\ge d$.  
\end{proof}

The next lemma is a mild generalization of \cite[Proposition 3.1]{BFMW} to arbitrary 
fields. The latter proposition is restricted to the rational field $\mathbb Q$, and we 
would be remiss were we not to record that Schinzel \cite[Theorem 22]{Sch2000} attributes 
this conclusion to Capelli.

\begin{lem}\label{lem:extension}
Let $K$ be a field. Suppose that $f(x)\in K[x]$ is a monic irreducible polynomial, let 
$\alpha$ be a root of $f$ lying in a splitting field extension for $f$ over $K$, and 
put $L=K(\alpha)$. Then, for any non-constant polynomial $g(t)\in K[t]$, the 
polynomial $f(g(t))$ is reducible in $K[t]$ if and only if $g(t)-\alpha$ is reducible 
in $L[t]$.
\end{lem}

\begin{proof}
We consider the $K$-algebra $A = K[x,t]/(f(x),g(t)-x)$ from two perspectives. First, 
on noting that $f(x)$ is irreducible over $K[x]$, we find that 
$K[x]/(f(x))\isom K[\alpha]=K(\alpha)=L$. Thus, on the one hand,
\begin{equation*}
    A\isom \frac{K[x,t]/(f(x))}{(g(t)-x)}\isom L[t]/(g(t)-\alpha).
\end{equation*}
Here, of course, we view $(g(t)-x)$ as being an ideal in $K[x,t]/(f(x))$. On the other 
hand, similarly,
\begin{equation*}
    A\isom \frac{K[x,t]/(g(t)-x)}{(f(x))}\isom K[t]/(f(g(t))).
\end{equation*}
Thus, we obtain a $K$-algebra isomorphism 
\begin{equation}\label{eq:Kalgisom}
K[t]/(f(g(t)))\isom L[t]/(g(t)-\alpha).
\end{equation}
Hence $K[t]/(f(g(t)))$ is a field if and only if $L[t]/(g(t)-\alpha)$ 
is a field, and thus $f(g(t))$ is irreducible in $K[t]$ if and only if $g(t)-\alpha$ is 
irreducible in $L[t]$. The desired conclusion follows.
\end{proof}

We take the opportunity to record a further consequence of the relation 
\eqref{eq:Kalgisom}, since it may be of use in future investigations concerning 
superirreducibility.

\begin{lem}
Let $K$ be a field. Suppose that $f(x)\in K[x]$ is a monic irreducible polynomial, and 
let $g(t)\in K[t]$ be any non-constant polynomial. Then, for any polynomial divisor 
$h(t)$ of $f(g(t))$, we have $\deg(f)|\deg(h)$.
\end{lem}

\begin{proof}
The relation \eqref{eq:Kalgisom} shows that $K[t]/(f(g(t))$ has the structure of an 
$L$-algebra. Any ring quotient of an $L$-algebra is still an $L$-algebra. Thus, we see 
that $K[t]/(h(t))$ is an $L$-algebra, and in particular a vector space over $L$. 
Consequently, one has 
\begin{equation*}
\deg(h)=\dim_K K[t]/(h(t))=[L:K]\left(\dim_L K[t]/(h(t))\right) 
=\deg(f)\left( \dim_L K[t]/(h(t)) \right),
\end{equation*}
and thus $\deg(f)|\deg(h)$.
\end{proof}

We also provide a trivial lemma explaining the relationship between our definitions of superirreducibility and weak superirreducibility for different values of $k$.
\begin{lem}\label{lem:weak}
    Let $R$ be a commutative domain with unity, and let $f(x) \in R[x]$ and $k \in \N$. 
    The polynomial $f(x)$ is $k$-superirreducible if and only if it is weakly 
    $\ell$-superirreducible for all natural numbers $\ell \leq k$. The polynomial $f(x)$ is weakly $k$-superirredcubible if and only if it is weakly $\ell$-superirreducible for all natural numbers $\ell$ dividing $k$.
\end{lem}
\begin{proof}
    All of the implications follow formally from the definitions except for the statement that, if $f(x)$ is weakly $k$-superirreducible and $\ell|k$, then $f(x)$ is weakly $\ell$-superirreducible. To prove this, write $k=\ell m$ and consider a polynomial $g(t)$ of degree $\ell$. The substitution $f(g(t^m))$ is thus irreducible, and hence so is $f(g(t))$.
\end{proof}
It follows that ``$2$-superirreducible'' and ``weakly $2$-superirreducible'' are synonyms.

\section{Counting $2$-superirreducible polynomials over finite fields}
Recall that when $1\le k<d$, we write $s_k(q,d)$ for the number of monic weakly 
$k$-superirreducible polynomials lying in $\F_q[t]$ having degree $d$. 
In particular, the concluding remark of the previous section shows that $s_2(q,d)$ is the number 
of monic $2$-superirreducible polynomials in $\F_q[t]$ having degree $d$. Our goal in this 
section is to establish formulae for $s_2(q,d)$ that deliver the conclusions recorded in 
\Cref{theorem1.1}.

\subsection{Elementary cases}
We begin by confirming that when $q$ is a power of $2$, and also when $d$ is odd, one 
has $s_2(q,d)=0$. In fact, rather more is true, as we now demonstrate.

\begin{prop}\label{char2}
Let $p$ be a prime. Then for all natural numbers $\ell$ and $d$, one has $s_p(p^\ell,d)=0$.
\end{prop}

\begin{proof}
Consider a polynomial $f\in \F_{p^\ell}[t]$ having degree $d$. Write 
$f(x)=\sum_{j=0}^d a_jx^j$, and note that $a_j=a_j^{p^\ell}$ for each index $j$. Thus, 
we have
\begin{equation*}
f(t^p)=\sum_{j=0}^d a_j^{p^\ell}t^{pj}=
{\left(\sum_{j=0}^d a_j^{p^{\ell-1}} t^j\right)\!}^p,
\end{equation*}
and it follows that $f(x)$ is not weakly $p$-superirreducible. Consequently, one has 
$s_p(p^\ell,d)=0$.
\end{proof}

The special case $p=2$ of \Cref{char2} shows that $s_2(q,d)=0$ when $q$ is a power of $2$. 
Next, we turn to polynomials of odd degree over $\F_q$.

\begin{prop}\label{oddd}
When $d$ is an odd natural number, one has $s_2(q,d)=0$.
\end{prop}

\begin{proof}
In view of the case $p=2$ of \Cref{char2}, there is no loss of generality in assuming that 
$q$ is odd. Let $f(x)\in \F_q[x]$ be a monic irreducible polynomial of degree $d$. The 
polynomial $f$ has a root $\alpha$ lying in $\F_{q^d}$, and $\F_{q^d}=\F_q(\alpha)$. By 
virtue of \Cref{lem:extension}, if we are able to find a quadratic polynomial 
$g(t)\in \F_q[t]$ having the property that $g(t)-\alpha$ has a root in $\F_{q^d}$, then we 
may infer that $f(g(t))$ is reducible. This will confirm that $f(x)$ is not 
$2$-superirreducible, delivering the desired conclusion.\par 

We may divide into two cases:

\begin{itemize}
\item[(a)] Suppose first that $\alpha=\beta^2$ for some $\beta\in\F_{q^d}$. Then we put 
$g(t)=t^2$ and observe that the polynomial $g(t)-\alpha$ has the root $\beta\in \F_{q^d}$.
\item[(b)] In the remaining cases, we may suppose that $\alpha$ is not the square of any 
element of $\F_{q^d}$. Since $q\ne 2$, there exists an element $b\in \F_q$ which is not the 
square of any element of $\F_q$. On recalling our assumption that $d$ is odd, we find that 
$b$ is not the square of any element in $\F_{q^d}$. Thus, we may infer that 
$b^{-1}\alpha=\beta^2$ for some $\beta\in \F_{q^d}$. We now put $g(t)=bt^2$ and observe 
that the polynomial $g(t)-\alpha$ has the root $\beta\in \F_{q^d}$.
\end{itemize}

In either case, our previous discussion shows that $f(x)$ is not
$2$-superirreducible, and this implies the desired conclusion.
\end{proof}

The conclusion of \Cref{oddd} combines with that of \Cref{char2} to confirm the first 
assertion of \Cref{theorem1.1}. These cases of \Cref{theorem1.1} help to explain the example 
noted in the introduction demonstrating that weak $(k-1)$-superirreducibility is not 
necessarily inherited from the corresponding property of weak $k$-superirreducibility. 
Expanding a little on that example, we observe that by making use of commonly available 
computer algebra packages, one finds the following examples of polynomials weakly 
$3$-superirreducible over $\F_2[x]$ yet not $2$-superirreducible over $\F_2[x]$:
\begin{align*}
    x^6+x^5+x^3+x^2\;&+1,\\
    x^8+x^6+x^5+x^3\;&+1,\\
    x^{10}+x^9+x^7+x^2\;&+1,\\
    x^{10}+x^9+x^8+x^4+x^3+x^2\;&+1,\\
    x^{10}+x^9+x^7+x^6+x^5+x^4+x^3+x^2\;&+1.
\end{align*}
In each of these examples of a polynomial $f\in \F_2[x]$, the failure of 
$2$-superirreducibility follows from \Cref{char2}. Meanwhile, a direct computation 
confirms that the polynomial $f(g(t))$ is irreducible over $\F_2[t]$ for each of the $8$ 
possible monic cubic polynomials $g(t)$ lying in $\F_2[t]$. No analogous odd degree examples 
are available, of course, by virtue of \Cref{oddd}, though examples of larger even degrees 
are not too difficult to identify. 

\subsection{Heuristics}\label{ssec:heuristic}
We next address the problem of determining a formula for the number $s_k(q,d)$ of monic 
weakly $k$-superirreducible polynomials of degree $d$ over $\F_q$. The simplest situation 
here with $k=1$ is completely resolved by celebrated work of Gauss, since 
$1$-superirreducibility is equivalent to irreducibility. Thus, as is well-known, it follows 
from Gauss \cite[p.~602]{gauss} that
\begin{equation*}
s_1(q,d)=\frac{1}{d}\sum_{e|d}\mu \biggl(\frac{d}{e}\biggr) q^{e},
\end{equation*}
whence, as $d\rightarrow \infty$, one has the asymptotic formula
\begin{equation*}
s_1(q,d)=\frac{q^d}{d}+O\biggl( \frac{1}{d}q^{d/2}\biggr) .
\end{equation*}
The corresponding situation with $k\geq 2$ is more subtle. We now motivate our proof 
of an asymptotic formula for $s_2(q,d)$ with a heuristic argument that addresses the 
cases remaining to be considered, namely those where $d$ is even and $q$ is odd. The 
heuristic argument is based on the following lemma, which will also be used in the proof.

\begin{lem}\label{lem:quadshifts}
Let $q$ be an odd prime power, and let $f(x)\in \F_q[x]$ be a monic irreducible polynomial 
of even degree $d$. Let $\alpha \in \F_{q^d}$ be a root of $f(x)$. The polynomial $f(x)$ is 
$2$-superirreducible if and only if $\alpha+c$ is not a square in $\F_{q^d}$ for all 
$c\in \F_q$. 
\end{lem}

\begin{proof}
As a consequence of \Cref{lem:extension}, the polynomial $f(x)$ is 
$2$-superirreducible in $\F_q[x]$ if and only if $g(t)-\alpha$ is irreducible in 
$\F_{q^d}[t]$ for all quadratic polynomials $g\in \F_q[t]$. Since this condition is 
invariant under all additive shifts mapping $t$ to $t+v$, for $v\in \F_q$, it suffices to 
consider only the quadratic polynomials of the shape $g(t)=at^2-b$, with $a,b\in \F_q$. 
Moreover, the assumption that $d$ is even ensures that $a$ is a square in $\F_{q^d}$, and 
hence we may restrict our attention further to polynomials of the shape $g(t)=t^2-c$ with 
$c\in \F_q$. So $f(x)$ is $2$-superirreducible if and only if the equation 
$t^2-c=\alpha$ has no solution in $\F_{q^d}$ for any $c\in\F_q$.
\end{proof}

For heuristic purposes, we now model the behaviour of these elements $\alpha+c$ as if they 
are randomly distributed throughout $\F_{q^d}$. Since roughly half the elements of 
$\F_{q^d}$ are squares, one should expect that the condition that $\alpha+c$ is not a 
square is satisfied for a fixed choice of $c$ with probability close to $\frac{1}{2}$. 
Treating the conditions for varying $c\in \F_q$ as independent events, we therefore expect 
that $f(x)$ is $2$-superirreducible with probability close to $1/2^q$. Multiplying 
this probability by the number of choices for monic irreducible polynomials $f(x)$ of 
degree $d$, our heuristic predicts that when $d$ is even and $q$ is odd, one should have
\begin{equation*}
s_2(q,d)\approx \frac{q^d}{d 2^q}.
\end{equation*}
We shall see in the next subsection that this heuristic accurately predicts the asymptotic 
behaviour of $s_2(q,d)$ as $d\to \infty$ through even integers $d$.

\subsection{The large $d$ limit}\label{ssec:larged}
The asymptotic formula predicted by the heuristic described in the previous subsection will 
follow in the large $d$ limit from Weil's resolution of the Riemann hypothesis for curves 
over finite fields. We make use, specifically, of the Weil bound for certain higher 
autocorrelations of the quadratic character generalizing Jacobi sums. Our goal in this 
subsection is the proof of the estimate for $s_2(q,d)$ supplied by the following theorem, 
an immediate consequence of which is the asymptotic formula \eqref{1.1} supplied by \Cref{theorem1.1}.

\begin{thm}\label{thm:larged}
When $q$ is odd and $d$ is even, one has
\begin{equation*}
\left| s_2(q,d)-\frac{q^d}{d 2^q} \right| < \frac{q}{2d}q^{d/2}.
\end{equation*}
\end{thm}

The proof of this estimate is based on a more rigorous version of the heuristic argument given 
in \Cref{ssec:heuristic}, and it employs character sums that we now define.

\begin{defn}\label{autocor}
Let $q$ be an odd prime power, and write $\chi_q$ for the nontrivial quadratic character 
$\chi_q : \F_q^\times \to \{1,-1\}$, extended to $\F_q$ by setting $\chi_q(0)=0$. 
We define the \textit{order $n$ autocorrelation of $\chi_q$} with offsets 
$u_1,\ldots,u_n \in \F_q$ to be the sum
\begin{equation*}
a_q(u_1,\ldots,u_n)=\sum_{\beta \in \F_q} \chi_q(\beta+u_1)\cdots\chi_q(\beta+u_n).
\end{equation*}
Noting that this definition is independent of the ordering of the arguments, when 
$U=\{u_1,\ldots ,u_n\}$ is a subset of $\F_q$, we adopt the convention of writing $a_q(U)$ for 
$a_q(u_1,\ldots ,u_n)$.
\end{defn}

Note that $a_q(U)\in \Z$ for all subsets $U$ of $\F_q$. When $\abs{U}=1$ it is apparent 
that $a_q(U)=0$. Meanwhile, in circumstances where $\abs{U}=2$, so that $U=\{u_1,u_2\}$ for 
some elements $u_1,u_2\in \F_q$ with $u_1\neq u_2$, the autocorrelation 
$a_q(U)=a_q(u_1,u_2)$ is a 
quadratic Jacobi sum. Thus, in this situation, we have $a_q(u_1,u_2)=\pm 1$; see 
\cite[Chapter 8]{irelandrosen}. The higher-order correlations become more complicated, but 
we will see that they can easily be bounded. First, we relate the autocorrelations of 
$\chi_q$ to the number $s_2(q,d)$ of monic $2$-superirreducible polynomials of 
degree $d$ in $\F_q[x]$.

\begin{prop}\label{prop:exactcount}
Let $q$ be an odd prime power and $d$ be even. Then
\begin{equation*}
s_2(q,d)=\frac{1}{d 2^q}\sum_{\substack{e|d \\ {\text{$d/e$ {\rm odd}}}}} 
\mu\Bigl( \frac{d}{e}\Bigr)\biggl(q^{e} +\sum_{\emptyset \neq U\subseteq \F_q}
(-1)^{\abs{U}}a_{q^{e}}(U)\biggr).
\end{equation*}
\end{prop}

\begin{proof}
Consider a monic irreducible polynomial $f(x)\in \F_q[x]$ of degree $d$, and let $\alpha$ 
be a root of $f(x)$ in $\F_{q^d}$. It follows from \Cref{lem:quadshifts} that $f(x)$ is 
$2$-superirreducible if and only if $\alpha+c$ is not a square in $\F_{q^d}$ for 
each $c\in\F_{q}$. Since the latter condition is equivalent to the requirement that 
$\chi_{q^d}(\alpha + c)=-1$ for all $c\in \F_q$, we see that
\begin{equation*}
\prod_{c\in \F_q}\foh \left(1-\chi_{q^d}(\alpha +c)\right)=\begin{cases}1,&
\text{if $f$ is $2$-superirreducible},\\
0,&\text{otherwise}.
\end{cases}
\end{equation*}
This relation provides an algebraic formulation of the indicator function for
$2$-superirreducibility. Instead of summing this quantity over monic irreducible 
polynomials, we can instead sum over elements $\alpha\in\F_{q^d}$ not lying in any proper 
subfield, dividing by $d$ to account for overcounting. Thus, we find that
\begin{equation*}
s_2(q,d)=\frac{1}{d}\sum_{\substack{\alpha \in \F_{q^d}\\ 
\text{$\alpha \nin \F_{q^{e}}$ (\text{$e<d$ and $e|d$})}}} \prod_{c\in \F_q}\foh 
\left(1-\chi_{q^d}(\alpha + c)\right).
\end{equation*}
The condition on $\alpha$ in the first summation of this relation may be encoded using the 
M\"obius function. Thus, we obtain
\begin{equation*}
s_2(q,d)=\frac{1}{d 2^q}\sum_{e|d}\mu\Bigl( \frac{d}{e}\Bigr) \sum_{\alpha \in \F_{q^{e}}} 
\prod_{c\in \F_q}\left( 1-\chi_{q^d}(\alpha + c)\right).
\end{equation*}
When $d/e$ is even, the quadratic character $\chi_{q^d}$ on $\F_{q^d}$ restricts to the 
trivial character on $\F_{q^{e}}$, and when $d/e$ is odd, it instead restricts to 
$\chi_{q^{e}}$. We therefore deduce that
\begin{equation*}
s_2(q,d)=\frac{1}{d2^q}\sum_{\substack{e|d\\ \text{$d/e$ odd}}} \mu\Bigl( \frac{d}{e}\Bigr) 
\sum_{\alpha \in \F_{q^{e}}}\prod_{c\in \F_q}\left(1-\chi_{q^{e}}(\alpha + c)\right),
\end{equation*}
and on observing that
\begin{equation*}
\sum_{\alpha \in \F_{q^{e}}}\prod_{c\in \F_q}\left(1-\chi_{q^{e}}(\alpha + c)\right) 
=q^{e} + \sum_{\emptyset \neq U \subseteq \F_q} (-1)^{\abs{U}} a_{q^{e}}(U),
\end{equation*}
the desired formula for $s_2(q,d)$ follows.
\end{proof}

We next establish a bound on the autocorrelations $a_{q^e}(U)$.

\begin{lem}\label{lem:aasymp}
Let $q$ be an odd prime power. Suppose that $U$ is a non-empty subset of $\F_q$ with 
$\abs{U}=n$. Then for each positive integer $e$, one has 
$\abs{a_{q^e}(U)}\leq (n-1)q^{e/2}$.
\end{lem}

\begin{proof}
Observe that
\begin{equation*}
a_{q^e}(U)=\sum_{\beta\in\F_{q^e}}\chi_{q^e}(h(\beta)),
\end{equation*}
where $h(t)=(t+u_1)\cdots (t+u_n)$ is a polynomial in $\F_q[t]$ having roots 
$-u_1,\ldots ,-u_n$. Since $u_1,\ldots ,u_n$ are distinct and $\chi_{q^e}$ is a 
multiplicative character of order $2$, it follows from a version of Weil's bound 
established by Schmidt that 
$\abs{a_{q^e}(U)}\leq (n-1)q^{e/2}$; see \cite[Chapter 2, p.~43, Theorem 2C']{Sch1976}.
\end{proof}

Now we complete the proof of \Cref{thm:larged}. In this proof, we expend a little extra 
effort to achieve a more attractive conclusion.

\begin{proof}[Proof of \Cref{thm:larged}]
We begin by observing that, in view of \Cref{lem:aasymp}, one has
\begin{align}
\biggl| \sum_{\emptyset \neq U \subseteq \F_q} (-1)^{\abs{U}} a_{q^{e}}(U)\biggr| 
&\leq \sum_{n=1}^q \binom{q}{n}(n-1)q^{e/2}\notag \\
&=q^{e/2}\biggl( q\sum_{n=2}^q \binom{q-1}{n-1}-\sum_{n=2}^q \binom{q}{n}\biggr) \notag \\
&=q^{e/2}\bigl( q(2^{q-1}-1)-(2^q-q-1)\bigr).\label{tw1}
\end{align}
We note next that since $d$ is assumed to be even, then whenever $e$ is a divisor of $d$ 
with $d/e$ odd, it follows that $e$ is even. Moreover, if it is the case that $e<d$, then $e\le d/3$. 
The first constraint on $e$ here conveys us from \eqref{tw1} to the upper bound
\begin{align*}
\sum_{\substack{e|d\\ \text{$d/e$ odd}}}\abs{ \sum_{\emptyset \neq U \subseteq \F_q} 
(-1)^{\abs{U}} a_{q^{e}}(U)} 
&\leq \bigl( 2^{q-1}(q-2)+1\bigr) \sum_{m=0}^{d/2}q^m \\
&<\frac{q}{q-1}\bigl( 2^{q-1}(q-2)+1\bigr) q^{d/2}.
\end{align*}
Meanwhile, making use also of the second constraint on $e$, we obtain the bound
\begin{equation*}
\sum_{\substack{e|d\\ \text{$e<d$ and $d/e$ odd}}}q^{e}\le \sum_{0\le m\le d/3}q^m< 
\frac{q}{q-1}q^{d/2}.
\end{equation*}
By applying these bounds in combination with \Cref{prop:exactcount}, we deduce that
\begin{equation*}
\abs{s_2(q,d)-\frac{q^d}{d 2^q}}<\frac{1}{d 2^q}
\bigl( (q-1)2^{q-1}-2^{q-1}+2\bigr) \frac{q}{q-1}q^{d/2}\le \frac{q}{2d}q^{d/2}.
\end{equation*}
This completes the proof of \Cref{thm:larged}.
\end{proof}

\subsection{Vanishing in the large $q$ limit}
We turn our attention next to the behaviour of $s_2(q,d)$ when $d$ is fixed and $q$ is 
large. It transpires that $s_2(q,d)=0$ for large enough prime powers $q$. This conclusion 
follows from \Cref{lem:quadshifts} once we confirm that for every primitive element 
$\alpha \in \F_{q^d}$, there exists an element $c\in \F_q$ for which 
$\chi_{q^d}(\alpha+c)=1$.

\begin{lem}\label{lem:largeqcharsum}
Suppose that $q$ is an odd prime power and $\alpha \in \mathbb{F}_{q^d}$ is a primitive 
element. Then, whenever $q>(d-1)^2$, one has
\begin{equation*}
\biggl| \sum_{c\in\F_q}\chi_{q^d}(\alpha+c)\biggr| <q.
\end{equation*}
\end{lem}

\begin{proof}
Consider the $d$-dimensional commutative $\F_q$-algebra $\F_{q^d}=\F_q[\alpha]$. 
Observe that the character $\chi_{q^d}$ is not trivial on 
$\F_q[\alpha]=\F_{q^d}$. It follows from Wan \cite[Corollary 2.2]{wan}, taking $\beta = -\alpha$ in the notation of that paper, that
\begin{equation*}
\biggl| \sum_{c\in \F_q}\chi_{q^d}(c+\alpha)\biggr| \leq (d-1)q^{1/2}.
\end{equation*}
Provided that $q>(d-1)^2$, one has $(d-1)q^{1/2}<q$, and thus the desired conclusion 
follows.
\end{proof}

We are now equipped to establish the final conclusion of \Cref{theorem1.1}.

\begin{thm}\label{thm:largeq}
Let $d$ be an even integer, and suppose that $q$ is an odd prime power with $q>(d-1)^2$. 
Then $s_2(q,d)=0$.
\end{thm}

\begin{proof}
Suppose that $f(x)\in \F_q[x]$ is a $2$-superirreducible polynomial of degree $d$ 
over $\F_q$, and consider a root $\alpha \in \F_{q^d}$ of $f$. By \Cref{lem:quadshifts}, we 
must have $\chi_{q^d}(\alpha +c)=-1$ for every $c\in\F_q$, and hence
\begin{equation*}
\sum_{c\in \F_q}\chi_{q^d}(\alpha +c)=-q.
\end{equation*}
This contradicts the estimate supplied by \Cref{lem:largeqcharsum}, since we have assumed 
that $q>(d-1)^2$. Consequently, there can be no $2$-superirreducible polynomials of 
degree $d$ over $\F_q$. 
\end{proof}

\section{Relationship to rational and $p$-adic superirreducibility}\label{sec:4}
Fix a rational prime number $p$. Then, any monic polynomial $f\in \Z[x]$ that is 
irreducible modulo $p$ is also irreducible over $\Q[x]$. One might guess that this familiar 
property extends from irreducibility to superirreducibility. Thus, if the monic polynomial 
$f(x)$ reduces to a weakly $k$-superirreducible polynomial modulo $p$, one might expect 
that $f(x)$ is itself weakly $k$-superirreducible over $\Z$, and perhaps also over $\Q$. We 
find that such an expectation is in fact excessively optimistic. Indeed, there are 
$2$-superirreducible polynomials over $\F_3$ with integral lifts that are not 
$2$-superirreducible over $\Z$.

\begin{eg}\label{example4.1}
Consider the polynomial $f(x) \in \Z[x]$ given by
\begin{equation*}
f(x)=x^4 -12x^3 +2x^2 -39x +71.
\end{equation*}
Then, we have $f(x)\equiv x^4 -x^2 -1\Mod{3}$, and it is verified by an exhaustive check 
that $x^4 -x^2 -1$ is $2$-superirreducible in $\F_3[x]$. However, one has
\begin{equation*}
f(3t^2 +t) =(t^4+3t^3+2t^2-1)(81t^4-135t^3-27t^2+39t-71),
\end{equation*}
so that $f(x)$ is not $2$-superirreducible over $\Z$.
\end{eg}

Despite examples like the one above, one may still hope that the assumption of additional 
congruential properties involving higher powers of $p$ might suffice to exclude such 
problematic examples, thereby providing a means to lift superirreducible polynomials over 
$\Z_p$ to superirreducible polynomials over $\Z$. The following proposition reveals a major 
obstruction to any such lifting process, since it shows that for each natural number 
$k\ge 2$, there are no $p$-adic weakly $k$-superirreducible polynomials.

\begin{prop}\label{proposition4.2}
Let $p$ be a prime number. When $k\ge 2$, there are no weakly $k$-superirreducible 
polynomials over $\Z_p$ or over $\Q_p$.
\end{prop}

\begin{proof} Suppose, if possible, that $f\in \Q_p[x]$ is a weakly $k$-superirreducible 
polynomial. There is no loss of generality in assuming that $f$ is an irreducible 
polynomial lying in $\Z_p[x]$. Let $\alpha$ be a root of $f$ lying in a splitting field 
extension for $f$ over $\Q_p$, and let $e=1+|v_p(\alpha)|$, where $v_p(\alpha)$ is defined 
in such a manner that $|\alpha|_p=p^{-v_p(\alpha)}$. Let $h\in \Z_p[t]$ be any polynomial 
of degree $k$, put $g(t)=p^e h(t)+t$, and consider the equation $g(\beta)=\alpha$. Since 
$|g(\alpha)-\alpha|_p<1$ and $|g'(\alpha)|_p=|1+p^e h'(\alpha)|_p=1$, an application of 
Hensel's lemma demonstrates that the equation $g(\beta)=\alpha$ has a solution 
$\beta\in \Q_p(\alpha)$. Thus, the equation $\alpha =p^eh(\beta)+\beta$ has a solution 
$\beta\in \Q_p(\alpha)$, and by appealing to \Cref{lem:extension}, we conclude that the 
polynomial $f(p^e h(t)+t)$ is reducible over $\Q_p[t]$. Since $p^eh(t)+t\in \Z_p[t]$, we 
see that $f$ is neither weakly $k$-superirreducible over $\Z_p$ nor over $\Q_p$, and we 
arrive at a contradiction. The desired conclusion follows.
\end{proof}

The discussion of this section appears to show, therefore, that superirreducibility over 
$\F_p$, and indeed superirreducibility over $\Z_p$ and $\Q_p$, is not closely connected to 
corresponding superirreducibility over $\Z$ and $\Q$.

\section{Acknowledgements}
We thank an anonymous referee for providing helpful comments on improving the clarity of the paper as well as for bringing to our attention the references \cite{benedetto, jonesandlevy, odoni}. We are grateful to the developers of Magma, Mathematica, and SageMath for writing software that we used to search for numerical examples and perform plausibility checks on asymptotics proved in this paper.\par

The fourth author is supported by NSF grant DMS-2302514. The fifth author is supported by NSF grants DMS-1854398 and DMS-2001549. The first, third, and fourth authors are also supported by the Heilbronn Institute for Mathematical Research.

\end{document}